\documentclass[12pt]{amsart}
\usepackage{mathrsfs}
\usepackage{txfonts}
\usepackage{amssymb}
\usepackage{amsmath,amssymb}
\newtheorem{theorem}{Theorem}[section]

\newtheorem{lemma}[theorem]{Lemma}

\theoremstyle{definition}

\theoremstyle{remark}

\numberwithin{equation}{section}

\begin{document}

\title[Limiting Weak Type Estimate for Capacitary Maximal Function]
{Limiting Weak Type Estimate for Capacitary Maximal Function}


\author[Jie Xiao]{Jie Xiao}
\address{Jie Xiao\\Department of Mathematics and Statistics, Memorial University of Newfoundland, St. John's, NL A1C 5S7, Canada}
\email{jxiao@mun.ca}
\author[Ning Zhang]{Ning Zhang}
\address{Ning Zhang\\Department of Mathematics and Statistics, Memorial University of Newfoundland, St. John's, NL A1C 5S7, Canada}
\email{nz7701@mun.ca}

\thanks{Project supported by NSERC of Canada as well as by URP of Memorial University, Canada.}

\subjclass[2000]{Primary 42B25; Secondary 46B70, 28A12}
\date{}

\keywords{Limiting weak type estimate; Capacitary maximal function}

\begin{abstract} A capacitary analogue of the limiting weak type estimate of P. Janakiraman for the Hardy-Littlewood maximal function of an $L^1(\mathbb R^n)$-function (cf. \cite{article1, article2}) is discovered.
\end{abstract}

\maketitle



\section{Statement of Theorem}\label{s1} 

For an $L^1_{loc}$-integrable function $f$ on $\mathbb R^n$, $n\ge 1$, let $Mf(x)$ denote the Hardy-Littlewood maximal function of $f$ at $x\in\mathbb R^n$:
$$
Mf(x)=\sup_{x\in B}\frac{1}{\mathscr{L}(B)}\int_B|f(y)|dy,
$$
where the supremum is taken over all Euclidean balls $B$ containing $x$ and $\mathscr{L}(B)$ stands for the $n$-dimensional Lebesgue measure of $B$. Among several results of \cite{article1, article2}, P. Janakiraman obtained the following fundamental limit:
$$
\lim _{\lambda\to 0} \lambda \mathscr{L}\big(\{x \in \mathbb{R}^n:\ Mf(x)> \lambda\}\big)=\|f\|_1=\int_{\mathbb R^n}|f(y)|dy\quad\forall\quad f\in L^1(\mathbb R^n).
$$

This note studies the limiting weak type estimate for a capacity. To be more precise, recall that a set function $C(\cdot)$ on $\mathbb{R} ^n$ is said to be a capacity (cf. \cite{article0, article3}) provided that
$$
\begin{cases}
C(\emptyset)=0;\\
0\leq C(A) \leq \infty\quad\forall\quad A\subseteq\mathbb R^n;\\
C(A)\leq C(B)\quad\forall\quad A\subseteq B\subseteq\mathbb R^n;\\
C(\cup _{i=1}^ {\infty}A_i)\leq \sum _{i=1} ^{\infty}C(A_i)\quad\forall\quad A_i\subseteq\mathbb R^n.
\end{cases}
$$
For a given capacity $C(\cdot)$ let 
$$
M_Cf(x)=\sup_{x\in B}\frac{1}{C(B)}\int _B |f(y)|dy
$$
be the capacitary maximal function of an $L^1_{loc}$-integrable function $f$ at $x$ for which the supremum ranges over all Euclidean balls $B$ containing $x$; see also \cite{article4}. 

In order to establish a capacitary analogue of the last limit formula for $f\in L^1(\mathbb R^n)$, we are required to make the following natural assumptions:

\begin{itemize}

\item Assumption 1 - the capacity $C\big(B(x,r)\big)$ of the ball $B(x,r)$ centered at $x$ with radius $r$ is a function depending on $r$ only, but also the capacity $C\big(\{x\}\big)$ of the set $\{x\}$ of a single point $x\in\mathbb R^n$ equals $0$.
\item Assumption 2 - there are two nonnegative functions $\phi$ and $\psi$ on $(0,\infty)$ such that 
$$
\begin{cases}
\phi(t)C(E)\leq C(tE)\leq \psi(t) C(E)\ \ \forall \ \ t>0\ \ \&\ \ tE=\{tx\in\mathbb R^n:\ x\in E\subseteq\mathbb R^n\};\\
\lim_{t\to 0} {\phi(t)}=0=\lim_{t\to 0}{\psi(t)}\ \ \&\ \ \lim_{t\to 0} {\psi(t)}/{\phi(t)}=\tau\in (0,\infty).
\end{cases}
$$
\end{itemize}
Here, it is worth mentioning that the so-called $p$-capacity satisfies all the assumptions; see also \cite{article5}.

\begin{theorem}\label{t11}
Under the above-mentioned two assumptions, one has
$$
\lim _{\lambda\to 0} \lambda C\big(\{x \in \mathbb{R} ^n: M_C f(x)> \lambda\}\big)\approx \|f\|_1\quad\forall\quad f\in L^1(\mathbb R^n).
$$
Here and henceforth, $\mathsf{X}\approx\mathsf{Y}$ means that there is a constant $c>0$ independent of $\mathsf{X}$ and $\mathsf{Y}$ such that $c^{-1}\mathsf{Y}\le\mathsf{X}\le c\mathsf{Y}$. 
\end{theorem}

\section{Four Lemmas}\label{s2}

To prove Theorem \ref{t11}, we will always suppose that $C(\cdot)$ is a capacity obeying Assumptions 1-2 above, but also need four lemmas based on the following capacitary maximal function $M_C \nu$ of a finite nonnegative Borel measure $\nu$ on $\mathbb R^n$:
$$
M_C\nu(x)=\sup_{B\ni x}\frac{\nu(B)}{C(B)}\quad\forall\quad x\in\mathbb R^n,
$$
where the supremum is taken over all balls $B\subseteq\mathbb R^n$ containing $x$.

\begin{lemma}\label{l21} If $\delta_0$ is the delta measure at the origin, then
$$
C\big(\{x \in \mathbb{R}^n: M_C \delta _0(x) > \lambda\}\big)=\frac{1}{\lambda}.
$$
\end{lemma}
\begin{proof} According to the defintion of the delta measure and Assumptions 1-2, we have
$$
M_C \delta _0(x)=\frac{1}{C(B(x,|x|))}\quad\forall\quad |x|\neq 0.
$$
Now, if $x$ obeys $M_C\delta_0(x)>\lambda$, then 
$$
C(B(x,|x|))< \frac{1}{\lambda}.
$$
Note that if $C\big(B(0,r)\big)$ equals $\frac{1}{\lambda}$, then one has the following property:
$$
\begin{cases}
C(B(x,|x|))< \frac{1}{\lambda}\quad\forall\quad |x|<r;\\
C(B(x,|x|))= \frac{1}{\lambda}\quad\forall\quad |x|=r;\\
C(B(x,|x|))> \frac{1}{\lambda}\quad\forall\quad |x|>r.\\
\end{cases}
$$
Therefore,
$$
\{x\in \mathbb{R}^n: M_C \delta _0(x)> \lambda\}=B(0,r),
$$
and consequently, 
$$
C\big(\{x\in\mathbb R^n:\   M_C \delta _0(x)> \lambda\}\big)=C\big(B(0,r)\big)=\frac{1}{\lambda}.
$$
\end{proof}

\begin{lemma}\label{l22} If $\nu$ is a finte nonnegative Borel measure on $\mathbb R^n$ with $\nu(\mathbb{R}^n)=1$, then
$$
\lim _{t\to 0} C\big(\{ x \in \mathbb{R}^n:\ M_C\nu_t(x)> \lambda\}\big)=\frac{1}{\lambda},
$$
where 
$$
\begin{cases}
t>0;\\
\nu_t(E)=\nu(\frac{1}{t}E);\\
\frac{1}{t}E=\{\frac{x}{t}:\ x\in E\};\\
E\subseteq\mathbb R^n.
\end{cases}
$$
\end{lemma}

\begin{proof} For two positive numbers $\epsilon$ and $\eta$, choose $\epsilon_1$ small relative to both $\epsilon$ and $\eta$, but also let $t$ be small and the induced $\epsilon_t$ be such that
$$
\begin{cases}
\nu _t\big(B(0,\epsilon_t)\big)>1-\epsilon;\\
\epsilon_t=3^{-1}\epsilon_1;\\
\lim_{t\to 0}\epsilon_t=0;\\
\epsilon<\eta C\big(B(0,\epsilon_1)\big).
\end{cases}
$$
Now, if
$$
\begin{cases}
E_{1,\lambda}^t=\Big\{x \in \mathbb R^n\setminus B(0,\epsilon_1):\ \lambda < M_C\nu_t(x)\leq \frac{1}{C\big(B(x,|x|-\epsilon_t)\big)}\Big\};\\
E_{2,\lambda}^t=\Big\{x \in\mathbb R^n\setminus B(0,\epsilon_1):\ \max\big\{\lambda,  \frac{1}{C\big(B(x,|x|-\epsilon_t)\big)}\big\}<M_C\nu _t(x)\Big\},
\end{cases}
$$
then 
$$
E_{1,\lambda}^t\cup E_{2,\lambda}^t \cup B(0,\epsilon_1)=\{ x \in \mathbb{R}^n: M_C \nu_t(x) > \lambda\}.
$$

On the one hand, for such $x\in E_{2,\lambda}^t$ and $\forall \tilde{r}>0$ that
$$
\frac{\nu_t\big(B(x,\tilde{r})\big)}{C\big(B(x,|x|-\epsilon_t)\big)}\leq \frac{1}{C\big(B(x,|x|-\epsilon_t)\big)} <M_C\nu _t(x).
$$
Additionally, since for any $r_1,\ r_2$ satisfying $0\leq r_1\leq r_2$,
$$
C\big(B(x,r_1)\big)\leq C\big(B(x,r_2)\big),
$$
$C\big(B(x,r)\big)$ is an increasing function with respect to $r$.
There exists $r< |x|-\epsilon_t$ such that 
$$
\frac{\nu_t\big(B(x,r)\big)}{C\big(B(x,|x|-\epsilon_t)\big)}\leq \frac{\nu_t\big(B(x,r)\big)}{C\big(B(x,r)\big)} \leq M_C\nu _t(x),
$$
and hence by the Assumption 1, for any $x_i\in E^t_{2,\lambda}$ there exists $r_i>0$ such that 
$$
r_i< |x_i|-\epsilon_t \ \ \&\ \ \lambda \leq \frac{\nu_t\big(B(x_i,r_i)\big)}{C\big(B(x,r)\big)}.
$$
By the Wiener covering lemma, there exists a disjoint collection of such balls $B_i=B(x_i,r_i)$ and a constant $\alpha>0$ such that 
$$
\cup_i B_i\subseteq E_{2,\lambda}^t \subseteq \cup_i \alpha B_i,
$$
Therefore, we get a constant $\gamma>0$, which only depends on $\alpha$, such that
$$
C(E_{2,\lambda}^t)\leq \gamma \sum_i C(B_i)< \gamma \sum_i \frac{\nu_t(B_i)}{\lambda}\leq \frac{\gamma \epsilon}{\lambda},
$$
thanks to
$$
B_i \cap B(0,\epsilon_t)= \emptyset\ \ \&\ \ 1-\nu _t\big(B(0,\epsilon_t)\big)<\epsilon.
$$

On the other hand, if $x\in E_{1,\lambda}^t$, then
\begin{eqnarray*}
\frac{1-\epsilon}{C\big(B(x,|x|+\epsilon_t)\big)}&\leq& \frac{\nu_t \big(B(x,|x|+\epsilon_t)\big)}{C\big(B(x,|x|+\epsilon_t)\big)}\\
&\leq& M_C\nu _t(x)\\
&\leq& \frac{1}{C\big(B(x,|x|-\epsilon_t)\big)}.
\end{eqnarray*}
Since 
$$
\begin{cases}
\lim_{t\to 0}\left(\frac{1}{C\big(B(x,|x|+\epsilon_t)\big)}-\frac{1}{C\big(B(x,|x|-\epsilon_t)\big)}\right)=0;\\
\lim_{t\to 0}\left(\frac{1}{C\big(B(x,|x|+\epsilon_t)\big)}-\frac{1}{C\big(B(x,|x|)\big)}\right)=0,
\end{cases}
$$
for $\eta>0$ there exists $T>0$ such that 
\begin{eqnarray*}
|M_C \nu_t(t)-M_C \delta_0|&<& \eta+\frac{\epsilon}{C\big(B(0,|x|)\big)}\\
&<& \eta+\frac{\epsilon}{C\big(B(0,\epsilon_1)\big)}\\
&<& 2 \eta \quad \forall\ t\in (0,T).
\end{eqnarray*}
Note that
$$
M_C \delta_0(x)-2 \eta \leq M_C \nu _t \leq M_C \delta_0(x)+2 \eta\quad\forall\quad x\in E_{1,\lambda}^t.
$$
Thus
$$
\{x\in\mathbb R^n:\ M_C \delta_0(x)> \lambda +2 \eta\} \subseteq E_{1,\lambda}^t\subseteq \{x\in\mathbb R^n:\ M_C \delta_0(x)> \lambda +2 \eta\}.
$$
This in turn implies
\begin{eqnarray*}
&&C\big(\{x\in\mathbb R^n:\ M_C \delta_0(x)> \lambda +2 \eta\}\big)\\
&&\quad \leq\ \ C(E_{1,\lambda}^t)\\
&&\quad \leq\ \ C\big(\{x\in\mathbb R^n:\ M_C \delta_0(x)> \lambda +2 \eta\}\big).
\end{eqnarray*}
Now, an application of Lemma \ref{l21} yields
$$
\frac{1}{\lambda+2 \eta}\leq C\Big(\{x\in\mathbb R^n:\ M_C\nu_t(x)>\lambda\}\cap\big(\mathbb R^n\setminus B(0,\epsilon_1)\big)\Big)\leq \frac{1}{\lambda-2 \eta}+\frac{\gamma \epsilon}{\lambda}.
$$
Letting $t\to 0$ and using Assumption 1, we get
$$
\lim _{t\to 0}C\big(\{x\in \mathbb{R}^n:\ M_C \nu_t(x) > \lambda\}\big)=\frac1{\lambda}.
$$
\end{proof}

\begin{lemma}\label{l23} If $\nu$ is a nonnegative Borel measure on $\mathbb R^n$, then $M_C \nu(x)$ is upper semi-continuous.
\end{lemma}
\begin{proof}
According to the definition of $M_C \nu(x)$, there exists a radius $r$ corresponding to $M_C \nu(x)>\lambda>0$ such that
$$
\frac{\nu(B(x,r))}{C(B(x,r))}>\lambda.
$$ 
For a slightly larger number $s$ with $\lambda+\delta>s> r$, we have
$$
\frac{\nu(B(x,r))}{C(B(x,s))}>\lambda.
$$
Then applying Assumption 1, for any $z$ satisfying $|z-x|< \delta$,
$$
M_C \nu(z) \geq \frac{\nu(B(z,s))}{C(B(z,s))} \geq \frac{\nu(B(x,r))}{C(B(x,s))}>\lambda.
$$
Thereby, the set $\{x\in\mathbb R^n: M_C \nu(x)>\lambda\}$ is open, as desired.
\end{proof}

\begin{lemma}\label{l24}  If $\nu$ is a finite nonnegative Borel measure on $\mathbb R^n$, then there exists a constant $\gamma>0$ such that
$$
\lambda C\big( \{x\in\mathbb R^n:\ M_C \nu(x)>\lambda\} \big)\leq \gamma \nu(\mathbb R^n).
$$
\end{lemma}

\begin{proof} Following the argument for \cite[Page 39, Theorem 5.6]{article6}, we set $E_{\lambda}=\{x\in\mathbb R^n:\  M_C \nu(x)>\lambda\}$, and then select a $\nu$-measurable set $E\subseteq E_{\lambda}$ with $\nu(E)<\infty$. Lemma \ref{l23} proves that $E_\lambda$ is open. Therefore, for each $x\in E$, there exists an $x$-related ball $B_x$ such that 
$$
\frac{\nu(B_x)}{C(B_x)}>\lambda.
$$
A slight modification of the proof of \cite[Page 39, Lemma 5.7]{article6} applied to the collection of balls $\{B_x\}_{x\in E}$, and Assumption 2, show that we can find a sub-collection of disjoint balls $\{B_i\}$ and a constant $\gamma>0$ such that
$$
C(E)\leq \gamma \sum_{i} C(B_i)\leq \sum_{i} \frac{\gamma}{\lambda}\nu(B_i)\leq \frac{\gamma}{\lambda} \nu(\mathbb R^n).
$$
Note that $E$ is an arbitrary subset of $E_{\lambda}$. Thereby, we can take the supremum over all such $E$ and then get
$$C(E_{\lambda})< \frac{\gamma}{\lambda}\nu(\mathbb R^n).$$
\end{proof}
\section{Proof of Theorem}\label{s3}

First of all, suppose that $\nu$ is a finite nonnegative Borel measure on $\mathbb R^n$ with $\nu(\mathbb R^n)=1$. According to the definition of the capacitary maximal function, we have
$$
M_C \nu_t(x)=\sup_{r>0} \frac{\nu_t(B(x,r))}{C(B(x,r))}=\sup_{r>0} \frac{\nu(B(\frac{x}{t},\frac{r}{t}))}{C(tB(\frac{x}{t},\frac{r}{t}))}.
$$
From Assumption 2 it follows that
$$
\frac{M_C \nu (\frac{x}{t})}{\psi(t)} \leq M_C \nu_t(x)\leq \frac{M_C \nu (\frac{x}{t})}{\phi(t)},
$$
and such that
\begin{eqnarray*}
\Big\{x\in\mathbb R^n:\  M_C\nu(\frac{x}{t})> \lambda \psi(t)\Big\}&\subseteq& \Big\{x\in\mathbb R^n:\  M_C\nu_t(x)> \lambda\Big\}\\
&\subseteq& \Big\{x\in\mathbb R^n:\  M_C\nu(\frac{x}{t})> \lambda \phi(t)\Big\}.
\end{eqnarray*}
The last inclusions give that
\begin{eqnarray*}
&&
\frac{\phi(t)}{\psi(t)}\lambda\psi(t) C\big(\{x\in\mathbb R^n:\  M_C\nu(x)> \lambda \psi(t)\}\big)\\
&&\quad\le\lambda \phi(t) C\big(\{x\in\mathbb R^n:\  M_C\nu(x)> \lambda \psi(t)\}\big)\\
&&\quad\leq\lambda C\big(\{tx\in\mathbb R^n:\  M_C\nu(x)> \lambda \psi(t)\}\big)\\
&&\quad=\lambda C\big(\{x\in\mathbb R^n:\  M_C\nu(x/t)> \lambda \psi(t)\}\big)\\
&&\quad\leq \lambda C\big(\{x\in\mathbb R^n:\  M_C\nu_t(x)> \lambda\}\big)\\
&&\quad\leq\lambda C\big(\{x\in\mathbb R^n:\  M_C\nu(x/t)> \lambda \phi(t)\}\big)\\
&&\quad=\lambda C\big(\{tx\in\mathbb R^n:\  M_C\nu(x)> \lambda \phi(t)\}\big)\\
&&\quad\leq \lambda \psi(t) C\big(\{x\in\mathbb R^n:\  M_C\nu (x)> \lambda \phi(t)\}\big)\\
&&\quad\leq \frac{\psi(t)}{\phi(t)} \lambda\phi(t) C\big(\{x\in\mathbb R^n:\  M_C\nu (x)> \lambda \phi(t)\}\big).
\end{eqnarray*}
These estimates and Lemma \ref{l22}, plus applying Assumption 2 and letting $t\to 0$, in turns derive 
\begin{eqnarray*}
\tau^{-1}&\le&\liminf_{\lambda \to 0} \lambda C\big(\{ x \in \mathbb{R}^n: M_C \nu(x) > \lambda\}\big)\\
&\le&\limsup_{\lambda \to 0} \lambda C\big(\{ x \in \mathbb{R}^n: M_C \nu(x) > \lambda\}\big)\le\tau.
\end{eqnarray*}

Next, let 
$$h(\lambda)=\lambda C\big(\{x\in \mathbb R^n: M_C \nu>\lambda\}\big).$$
By Lemma \ref{l24} and the last estimate for both the limit inferior and the limit superior, there exists two constants $A>0$ and $\lambda_0>0$ such that  
$$A\leq h(\lambda)\leq \gamma\quad\forall\quad \lambda\in (0,\lambda_0).$$
Moreover, for any given $\varepsilon>0$, choose a sequence $\{y_i=\big[\frac{\gamma}{A}(1-\varepsilon)^N\big]^i\}_1^{\infty}$, where $N$ is a natural number satisfying $\frac{\gamma}{A}(1-\varepsilon)^N<1$. Then, there exists an integer $N_0\ge 1$, such that $y_{N_0}<\lambda_0$. Hence, for any $n>m>N_0$ we have
\begin{eqnarray*}
&&|h(y_m)-h(y_n)|\\
&&\ \ \le\ |y_m C\big(\{x\in\mathbb R^n:\ \  M_C\nu(x)> y_m\}\big)-y_n C\big(\{x\in\mathbb R^n:\ \  M_C\nu(x)> y_n\}\big)|\\
&&\ \ \le\ |y_m-y_n| C\big(\{x\in\mathbb R^n:\ \  M_C\nu(x)> y_m\}\big)\\
&&\ \ \ +\  y_n |C\big(\{x\in\mathbb R^n:\ \  M_C\nu(x)> y_m\}\big)-C\big(\{x\in\mathbb R^n:\ \  M_C\nu(x)> y_n\}\big)|\\
&&\ \ \le\ |y_m-y_n| \frac{\gamma}{y_m}+y_n |\frac{\gamma}{y_n}-\frac{A}{y_m}|\\
&&\ \ \le\ \gamma(1-\frac{y_n}{y_m})+(\gamma-A\frac{y_n}{y_m})\\
&&\ \ \le\ \gamma(1-\big[\frac{\gamma}{A}(1-\varepsilon)^N\big]^{n-m}) +(\gamma-A\big[\frac{\gamma}{A}(1-\varepsilon)^N\big]^{n-m})\\
&&\ \ \le\ \gamma(1-(1-\varepsilon)^{N(n-m)}) +(\gamma-\gamma(1-\varepsilon)^{N(n-m)})\\
&&\ \ \le\ 2\gamma N(n-m)\varepsilon.
\end{eqnarray*}
Consequently, $\{h(y_i)\}$ is a Cauchy sequence, $D=\lim_{i \to \infty} h(y_i)$ exists. Note that for any small $\lambda$ there exists a large $i$ such that
$$y_{i+1}\leq \lambda\leq y_i.$$
Thereby, from the triangle inequality it follows that if $i$ is large enough then
\begin{eqnarray*}
|h(\lambda)-D| &\le& |h(\lambda)-h(y_i)|+|h(y_i)-D|\\
&\le& |y_i-\lambda| \frac{\gamma}{y_i}+\lambda |\frac{\gamma}{\lambda}-\frac{A}{y_i}|+|h(y_i)-D|\\
&\le& \gamma(1-\frac{\lambda}{y_i})+(\gamma-A\frac{\lambda}{y_i})+|h(y_i)-D|\\
&\le& \gamma(1-\frac{y_{i+1}}{y_i})+(\gamma-A\frac{y_{i+1}}{y_i})+|h(y_i)-D|\\
&\le& (2\gamma N+1)\varepsilon
\end{eqnarray*}
This in turn implies that $\lim_{\lambda \to 0} \lambda C\big(\{ x \in \mathbb{R}^n: M_C \nu(x) > \lambda\}\big)$ exists, and consequently,
$$
{\tau}^{-1} \leq \lim_{\lambda \to 0} \lambda C\big(\{ x \in \mathbb{R}^n: M_C \nu(x) > \lambda\}\big) \leq \tau
$$
holds.

Finally, upon employing the given $L^1(\mathbb R^n)$ function $f$ with $\|f\|_1>0$ to produce a finite nonnegative measure $\nu$ with $\nu(\mathbb{R}^n)=1$ via
$$
\nu(E)=\frac{1}{||f||_1}\int_E |f(y)|dy\quad\forall\quad E\subseteq\mathbb R^n,
$$
we obtain
$$
\lim _{\lambda \to 0} \lambda C\big(\{x\in \mathbb{R}^n:\ M_C f(x) > \lambda ||f||_1\}\big)\approx 1,
$$
thereby getting
$$
\lim _{\lambda \to 0} \lambda \|f\|_1 C\big(\{ x \in \mathbb{R}^n:\ M_C f(x) > \lambda \|f\|_1\}\big)\approx\|f\|_1.
$$
By setting $\tilde{\lambda}=\lambda \|f\|_1$ in the last estimate, we reach the desired result.

\end{document}